\documentclass[12pt]{article}

\usepackage[english]{babel}
\usepackage{amsfonts}
\usepackage{epsfig}
\usepackage{color}
\usepackage{amsmath}
\usepackage{amsthm}
\usepackage{amssymb}

\newcommand{\ad}{\text{ad\,}}

\newcommand{\Supp}{\text{Supp}}

\newcommand{\bi}{\bar\imath}

\newcommand{\ba}{\begin{array}}
	\newcommand{\bal}{\begin{array}{l}}
		\newcommand{\be}{\begin{equation}}
			\newcommand{\beqa}{\begin{eqnarray}}
				\newcommand{\bl}{\begin{lem}}
					\newcommand{\bt}{\begin{teo}}
						
						\newcommand{\C}{\mathbb{F}}
						\newcommand{\ch}{\choose}

						\newcommand{\ea}{\end{array}}
					\newcommand{\ee}{\end{equation}}
				\newcommand{\eeqa}{\end{eqnarray}}
			\newcommand{\el}{\end{lem}}
		\newcommand{\et}{\end{teo}}

	\newcommand{\la}{\langle}
	
	\newcommand{\mc}{\mathcal}

	\newcommand{\N}{\mathbb{N}}
	\newcommand{\nc}{\normalcolor}

	\newcommand{\ra}{\rangle}

	\newcommand{\vnj}{\mathcal{V}^n_j}

	\newcommand{\Z}{\mathbb{Z}}

	\newtheorem{lem}{Lemma}
	\newtheorem{teo}{Theorem}
	\newtheorem{pro}{Proposition}
	\newtheorem{cor}{Corollary}
	
	\theoremstyle{definition}                 
	\newtheorem{defin}{Definition}            

\begin{document}
		
		\title{Computation of Normal Forms for Systems\\ with  Many Parameters}
		
		\author{Tatjana Petek$^{1,4}$ and Valery G.~Romanovski$^{1,2,3}$\footnote{Corresponding author, email: valerij.romanovskij@um.si} \\
			$^1${\it Faculty of Electrical Engineering and Computer Science,} \\ {\it University of Maribor,
				Koro\v ska cesta 46, SI-2000 Maribor, Slovenia}\\
			$^2${\it Center for Applied Mathematics and Theoretical Physics,}\\
			{\it Mladinska 3, SI-2000 Maribor, Slovenia}\\
			$^3${\it Faculty of Natural Science and Mathematics,}\\ {\it  University of Maribor,
				Koro\v ska cesta 160, SI-2000 Maribor, Slovenia}\\
			$^4${\it Institute of Mathematics, Physics and Mechanics,}\\
			{\it Jadranska 19, SI-1000 Ljubljana, Slovenia} 
		}
		\date{}

		\maketitle
		
		
		
		\begin{abstract}{
There are two ways to compute Poincar\'e-Dulac normal forms of systems of ODEs. Under the original approach used by Poincar\'e  the normalizing transformation 
is explicitly computed.  On each step, the normalizing  procedure requires the  substitution
of  a polynomial to a series.   Under the other  approach, a normal form is computed using  Lie transformations. 
In this case,   the changes of coordinates are performed   as actions of  certain infinitesimal generators. 
In both cases, on each step the homological equation is solved in  the vector space of polynomial 
vector fields  $\vnj$
 where  each component of the vector field   is a
		homogeneous polynomial of degree $j$.
		We present the third way  of computing normal forms of polynomial systems 
				of ODEs where the coefficients of all terms are parameters.  Although we use   Lie transforms 
				the  homological equation is solved not in  $\vnj$ but in the vector 
				space  of polynomial vector fields where each component is a homogeneous polynomial in the 
				parameters of the system.  It is shown that the space of the parameters 
				is a kind of dual space and,  the computation of normal forms can be performed 
				in the space of parameters treated as the space of generalized vector fields. 
				The approach  provides a simple way to parallelize the  normal form computations opening the way 
				to compute  normal forms up to higher order than under previously known two approaches. }
		\end{abstract}

		\section{Introduction}

		The normal form theory of ordinary differential equations is a powerful tool for studying the qualitative behavior of nonlinear systems near equilibrium points and local integrability. The central idea of the  theory, which originates from the works of 
		Poincar\'e, Dulac and Lyapunov  \cite{P,Dul,Lya},   is to
		perform a change of coordinates or a succession of coordinate transformations, 
		to place the original system  into a form most amenable to study.
		
		There are many  books devoted to various aspects of the   normal form theory (see e.g. \cite{Bibikov,Bru, CDW,HY,Mur,Nay}). 
		Some discussions on the  recent progress can be found in the survey papers \cite{AFG,S}.
		
		Three main questions in the theory of normal form are to analyze the structure of normal form, compute it and 
		analyze the convergence of the obtained series. Our paper concerns only  algorithms for  computing  
		of normal forms  regardless of  their structure and convergence. 
		Usually,  computations of normal forms require
		huge amount of algebraic operations, so the algorithms used in normal form theory are  complex and computationally expensive.
		The complexity of the algorithms can make it challenging to compute normal forms for systems with a large number of variables.
		The complexity of 
		computations fast increases for equations dependent on parameters. For these reasons, 
		normal form computations   rarely can be performed by hand but are carried out using computer algebra systems. 
		
		There are two main approaches to computations of normal form. The classical one, used by Poincar\'e, Dulac and Lyapunov 
		requires substitution series into series. It appears that the most computationally efficient algorithms obtained using this 
		approach are given in \cite{Pei,TY}.  The other one is based on Lie transformations and was developed 
		starting from the works of Birkhoff \cite{Bir}, Steinberg \cite{Ste},  Chen \cite{Che} and others. Using this way a normalization is performed 
		as a sequence of linear transformations.

		Throughout this paper, fix $\mathbb{F}$ to be  either the field of all real numbers  or the field of all complex numbers. Let $
		\mathcal{V}^n$ be
		the  space of  vector fields $v : \mathbb{F}^n \to  \C^n$ 
		which coordinates are power series in $x_1, \dots , x_n$ (for brevity we will say "in $x$") vanishing at the origin, 
		which can be convergent or merely formal,   and  $\vnj$ be  
		the vector space of polynomial vector fields on $\C^n$, that is, vector
		fields $v : \C^n \to \C^n $  such that each component $v_i$, $i = 1, \dots , n$,  of $v$  is a
		homogeneous polynomial of degree $j$. Note that $
		\mathcal{V}^n$ and  $\vnj$ are vector spaces over $\C$ and $\mathcal{V}^n=\oplus_{j=0}^\infty\vnj$. 
		
		Consider the differential equation 
		\be\label{eq_u}
		\dot x=  u(a, x), 
		\ee
		where $a=(a_1, \dots, a_\ell ) $ is an $\ell$-tuple of parameters (regarded as element of $\C^\ell$), $x=(x_1,\dots, x_n)$,   $u(a,x)\in  \mathcal{V}^n$
		and all terms of $u(a,x)$ depend polynomially   on parameters   $a_1,\dots,a_\ell $ (on $a$ for brevity). 
		
		The  space  $\mathcal{V}^n$ can also be viewed as a module over the ring of power series in $x_1, \dots, x_n$.  It 
		has natural grading by the degree of polynomials, that is, equation \eqref{eq_u} can be written 
		as 
		\be \label{eq_u_h}
		\dot x= \sum_{j=1}^\infty u_j(a,x),
		\ee
		where  $ u_j(a,x)$ is a homogeneous polynomial of degree $j$ in $x$, that is, $ u_j(a,x) \in \vnj$. As a rule, the normalization 
		of system  \eqref{eq_u} is performed according to this grading -- polynomials $u_j$ are changed into  a more
		amendable   
		form  for studying for $j=1,2,3,\dots$, consequently. 
		
		In this paper, we use another approach which is  based on the grading of the power series on the right hand side of 
		\eqref{eq_u} by the degree of polynomials  viewed  as polynomials  in the parameters of the system.  
		Indeed, since by our assumption each term of  $u(a,x)$ depends polynomially  on parameters $a$,
		each component of the vector field  $u(a,x)$ is a power series  in $a_1, \dots, a_\ell, x_1,\dots, x_n$, so     
		we can write \eqref{eq_u} in the form 
		\be \label{eq_a}
		\dot x=  \sum_{s=1}^\infty \bar u_s(a,x),
		\ee
		where $\bar u_s$ is a homogeneous polynomial of degree $s$  in the parameters  $a_1, \dots, a_\ell$ of the 
		system. 
		
		We will work with formal vector fields using  the grading of the  formal series 
		by polynomials which are  homogeneous  with respect to the parameters $a$ of $u(a,x)$ (generally, they are \emph{ not homogeneous} with respect to the variable $x$).  
		This idea looks unusual at first glance, but it   makes perfect  sense  since  when a vector field depends 
		on  many parameters, then it becomes more important 
		to have a tool to handle the monomials  with respect to  the parameters rather than with  
		respect to  the unknown functions $x_i$.
		
		However, the module  of formal  vector fields   with components from $\C[[a,x]]$   is too large. 
		We define  a submodule which is graded using the structure of the monomials on the right-hand side of 
		\eqref{eq_a}.  As a result, we will obtain an algorithm for computing  normal forms which has two advantages 
		with respect to the traditional ones:
		\begin{itemize}
			\item[--]  It allows  performing the parallelization  of normal forms computations. Namely, for equation \eqref{eq_u}, the terms of powers 
			series in a normal  form have the form    $p(a) x^r$, where $x^r$ is a resonant monomial and $p(a)=\sum_{\alpha\in Supp(p)} 
			p_\alpha a^\alpha$ is a polynomial in $a$. Using traditional methods, one can compute only the whole term  $p(a) x^r$.
			Using our technique, it is possible to compute any term $p_\alpha a^\alpha$ of $p(a)$ without computing 
			the whole polynomial $p(a)$.   Since the computation of normal form 
			using computer algebra systems is extremely time and memory consuming, this parallelization should allow 
			computing normal forms up to higher order than using  the traditional approaches. 
			
			\item[--] The obtained algorithm is more convenient for computations by hand than the ones known before because 
			it does not require the computation of derivatives and collection of similar terms, since only arithmetic 
			operations with numbers are used.
		\end{itemize}
		
		As we will see, the  space of parameters of the differential equations is a kind of dual space and all computations
		can be performed in this dual space.  
		
		The paper is organized as follows.  In the next section, we present some basics of the normal forms. 
		In Section 3   a module which we use for our normal form  computations is defined and 
		our  normalization algorithm is presented.  In Section 4 we introduce the notion of generalized  vector fields,
		reformulate the algorithm  in terms of such vector fields and give an example of the normal form computation.

		\section{Preliminaries}
		
		Consider the autonomous differential equation  (we can also say  $n$-dimensional autonomous system) 
		\be
		\label{ndim}
		\dot x =   f(x), \hspace{3mm} x \in  \mathbb{F}^n,
		\ee
		where 
		$$
		f(x) =Ax + \sum_{j=2}^\infty
		f_j(x),
		$$
		$f_j(x)$   is an $n$-dimensional vector valued homogeneous polynomial of degree $j$, that is,   
		$f_j(x) \in \mathcal{V}^n_j$ and $f(x)\in \mathcal{V}^n$. In this paper, we assume that 
		$A$ is a \emph{diagonal matrix} with the eigenvalues $\lambda_1, \dots, \lambda_n$,
		$x=(x_1, \dots, x_n)^T.$  
		
		Let 
		$$\lambda=(\lambda_1, \dots, \lambda_n)^T\in \C^n$$ be the column vector   of eigenvalues of $A$. Set $\N_0 = \N\cup \{0\}$.  For $\alpha=(\alpha_1,\dots,\alpha_n) \in \N_0^n$ we denote $\langle\lambda,\alpha\rangle=\sum_{i=1}^n \alpha_i \lambda_i$, 
		$|\alpha|=\alpha_1+\dots+\alpha_n$ and $x^\alpha= x_1^{\alpha_1}  x_2^{\alpha_2} \cdots x_n^{\alpha_n}$.
		Throughout this paper {\it $x$ is a column vector} and  {\it $x^\alpha$ is a monomial. }
		
		It is well known that  any   formal invertible  change of coordinates of the form 
		\be
		\label{yhy}
		x = y+ \tilde  h(y)= y +\sum_{j=2}^\infty \tilde h_j(y), 
		\ee
		with  $\tilde h_{j} \in 
		\mathcal{V}^n_{j}$, $j=2,3, \dots,$  
		brings \eqref{ndim} to a system of a similar form, 
		\be
		\label{linearni}
		\dot y = A y+\tilde g(y),
		\ee
		where 
		$$\tilde g(y)= \sum_{j=2}^\infty \tilde g_j(y), \qquad\tilde  g_j(y)\in \vnj,  \qquad j=2,3, \dots .
		$$ 
		
		\nc
		
		Let  $e_k$ be  the $n$--dimensional column vector with its $k$-th component equal to $1$ and the others all equal to zero.
		It is said that the monomial 
		$x^\alpha$  ($|\alpha|>1$) in $  e_k^T f(x) $, $k=1,\dots, n$ is a \emph{resonant} monomial if $k$ and $\alpha$ satisfy
		\be \label{res_c_e} 
		\langle\lambda,\alpha\rangle- \lambda_k=0.
		\ee
		A term of $f(x)$  is called \emph{resonant} if it is a   resonant monomial multiplied by a scalar from $\C$ (or a parameter) and  \emph{nonresonant} otherwise.
		Similarly, the monomial  $y^\alpha$  in $  e_k^T \tilde g(y) $  or  in $  e_k^T \tilde h(y) $
		is resonant if \eqref{res_c_e} holds.  
		
		\begin{defin} \label{def_nf}
			We say that system \eqref{linearni} is in the  \emph{Poincar\'e--Dulac normal form}  (or, simply in the normal form)  
			if $\tilde g(y)$ contains only resonant  terms. 
		\end{defin}
		
		By  Poincar\'e--Dulac theorem any system \eqref{ndim} can be brought to a normal form by a substitution \eqref{yhy} \cite{P,Dul}.
		Transformation \eqref{yhy}  which brings system \eqref{ndim} to a  Poincar\'e--Dulac normal form  is called  a  normalization or a normalizing transformation. In general, the normalization of system \eqref{linearni} is not unique. 
		However, 
		the normalization containing only nonresonant terms is unique.


		For $ v(x)\in \mathcal{V}^n $ we denote by $Dv(x)$ the $n\times n$ matrix of partial derivatives of $v(x)$. 
		The linear operator $\pounds : \mathcal{V}^n\to \mathcal{V}^n$  defined by 
		\be \label{ho}
		(\pounds v)(x) =Dv(x) Ax-A v(x)
		\ee
		is called the \emph{homological operator}. It is easy to see that $\pounds  (\vnj) \subseteq \vnj$,   so let
		$$
		\pounds_j : \vnj \to \vnj 
		$$
		be the restriction of $\pounds$ to $\vnj$. It is known (see e.g. \cite{Bibikov,Mur})  that 
		the $ \pounds_j$  is semisimple with the eigenvalues 
		$\beta_{\alpha,i}=\la \alpha, \lambda\ra -\lambda_i$ and (basis) eigenvectors $p_{i,\alpha}=e_i x^\alpha$,
		where $|\alpha|=j$ and $ i=1,\dots, n$. The kernel of $\pounds_j$ is then spanned by all monomials $e_i x^\alpha$, where $i$ and $\alpha$ satisfy $\beta_{\alpha,i}=0
		$.  Then we can decompose $\vnj$  as  
		$$
		\vnj={\rm im} \pounds_j \oplus \ker \pounds_j.
		$$
		According to this decomposition, we can write every $v_j \in \vnj$ uniquely as
		\be \label{eq:L}
		v_j = \pounds_j  \tilde h_j + \tilde g_j
		\ee 
		for some $\tilde h_j \in \vnj$, i.e. $ \pounds  \tilde h_j \in {\rm im} \pounds_j$ and $\tilde g_j \in \ker \pounds_j$ is the (natural) projection of $v_j$ onto $\ker \pounds_j$. In other words, in $\tilde g_j$ can only appear terms $e_i x^\alpha$ with $\la \alpha, \lambda\ra -\lambda_i=0$, thus being resonant terms.

		Assume that equation \eqref{ndim} has already been normalized to  order $j-1$, so \eqref{linearni} is 
		a normal form of \eqref{ndim} up to order $j-1$.  
		Then to perform the normalization to order $j$, one solves the equation 
		\be \label{comp_nf}
		\pounds_j \tilde h_j= f_j - \tilde g_j,
		\ee
		known as the homological equation, for  $\tilde h_j$ and   $\tilde g_j$ 
		(the solution is unique, providing we choose all resonant terms of $\tilde h_j$ to be zero). 
		With this  $\tilde g_j$  \eqref{linearni} is in the normal form to order $j$, 
		and to perform the next step of the normalization, one needs to do in \eqref{ndim}
		the transformation 
		$$
		x \to x+\tilde h_j(x) 
		$$
		and re-expand the series $f(x) $ on the right hand of \eqref{ndim} up to order $j+1$.
		
		The computational procedure described above  is simple, however, it requires substitutions of  polynomials 
		into series
		and, therefore, it is very  laborious   from the computational point of view. 
		
		Another setting for computing  the normal form is using the Lie brackets. 
		For two vector fields $v(x)$ and $w(x)$ in $\mathcal{V}^n$  
		$$[w,v]=D(v)w-D(w)v$$ 
		is the  Lie bracket.   For a given $w(x)\in \mathcal{V}^n$, the operator 
		$$
		\text{ad}\, w(x) =[w(x),\cdot] \, : \, {\mathcal{V}^n} \to   {\mathcal{V}^n}
		$$
		is the so-called \textit{adjoint} map. 
		In particular, $\pounds = \ad Ax$ and
		equation  \eqref{linearni} is in the normal form if and only if 
		$$
		(\ad Ax)\, \tilde g_k = 0,   { \ \rm  for\  all \ }  k=2, 3, \dots.
		$$ 
		
		Let 
		\be \label{s_b}
		\frac{dx}{ds}= b(x),  \qquad  b(x) \in   {\mathcal{V}^n}, \qquad
		b(0)=0,
		\ee
		and let $\psi_s(x)$ be the flow of \eqref{s_b} satisfying $\psi_0 (x)=x$.
		By the Theorem on Lie series for vector fields (see \cite[Theorem 4.4.8]{Mur} or \cite[p. 202]{Oli}) 
		after the transformation 
		\be \label{trx}
		x \to \psi_s(x)
		\ee
		the right-hand side of equation \eqref{ndim} is changed to 
		\be \label{LT}
		\psi_s'(x)^{-1}f(\psi_s(x))=f(x)+s (\ad b(x)) f(x)+ \frac{s^2}{2}  (\ad b(x))^2 f(x)+\dots.
		\ee
		Thus, if equation \eqref{ndim} is in the normal form form up to order $j-1$, to perform 
		the next step of the normalization, one solves the homological equation \eqref{comp_nf} for $\tilde h_j$ 
		and $\tilde g_j$  and 
		performs in \eqref{ndim} the transformation \eqref{trx} with $s=1$ and $\psi_s(x)$ being the flow of 
		$$
		\frac{dx}{ds}=\tilde h_j(x).
		$$
		
		Denoting 
		$\psi_1(x)$ by  $\exp(\tilde h_j)$ we can write 
		the normalization to order $m$  as   
		$$x=H_{m}(y) \circ \cdots \circ H_2 (y),$$
		where    $H_k=\exp({\tilde h_k}) $, $k=2, \dots, m$.
		
		An algorithmic description of this normalizing procedure can be found, for instance, in \cite{San,Wal}.
		


		\section{A new algorithm for computing  of normal forms}
		
		\subsection{A grading of the formal power  series module}
		
		Assume that the  terms of the function $f(x)$ in \eqref{ndim} depends polynomially  on parameters $a_1, \dots, a_\ell$
		and let   $a=(a_1, \dots, a_\ell)$,
		that is, the  terms of $f(x)=f(a,x)$ are of the form 
		\be \label{fax}
		f^{(i)}_{(\mu_1,\dots,\mu_\ell,\beta_1,\dots,\beta_n)}  a_1^{\mu_1}\cdots a_\ell^{\mu_\ell} x_1^{\beta_1}\cdots   x_n^{\beta_n} \qquad (i=1, \dots, n).
		\ee  
		Then the coefficients of the resonant monomials in the normal form are polynomials in parameters $a_1, \dots a_\ell$.

		As an example, we  consider     
		the following quite arbitrarily  chosen  differential equation 
		\be \label{ex}
		\begin{aligned}
			\dot x_1 = & \phantom{ -} x_1 + a^{(1)}_{10} x_1^2 +  a^{(1)}_{01} x_1 x_2 +  a^{(1)}_{-13} x_2^3=x_1(1 + a^{(1)}_{10} x_1 +  a^{(1)}_{01}  x_2 +  a^{(1)}_{-13} x_1^{-1} x_2^3),\\
			\dot x_2= &-x_2 +   a^{(2)}_{10} x_1 x_2  +  a^{(2)}_{01}  x_2^2 +  a^{(2)}_{02} x_2^3=x_2(-1 +   a^{(2)}_{10} x_1  +  a^{(2)}_{01}  x_2 +  a^{(2)}_{02} x_2^2).
		\end{aligned}
		\ee
		Let us remark that there is a  deep reason for the given notation of parameters which will be very important in the sequel.

		Computing  the normal form of \eqref{ex} up to order 5  
		we obtain 
		\be \label{nfg35}
		\dot x={\rm diag}(1,-1) x +g_2(x)+g_3(x)+g_4(x) +g_5(x), 
		\ee
		where $x=(x_1,x_2)^T$, $g_2(x)=g_4(x)=0$ and 
		\begin{align*}
			g_3(x) &=((- a^{(1)}_{10}  a^{(1)}_{01}  +  a^{(1)}_{01}  a^{(2)}_{10}) x_1^2 x_2, (- a^{(1)}_{01}  a^{(2)}_{10}  + 
			a^{(2)}_{01}  a^{(2)}_{10}) x_1 x_2^2)^T \\
			&=(x_1(- a^{(1)}_{01} a^{(1)}_{10}  +  a^{(1)}_{01}  a^{(2)}_{10}) x_1 x_2,\ x_2(- a^{(1)}_{01}  a^{(2)}_{10}  + 
			a^{(2)}_{01}  a^{(2)}_{10}) x_1 x_2)^T, \\
			g_5(x) &= (( (a^{(1)}_{01})^2 a^{(1)}_{10}  a^{(2)}_{10} +  a^{(1)}_{01} a^{(1)}_{10}  a^{(2)}_{01}  a^{(2)}_{10} - 
			2  (a^{(1)}_{01})^2  (a^{(2)}_{10})^2) x_1^3 x_2^2  , \\
			& \phantom{ -}(- a^{(1)}_{01} a^{(1)}_{10}  a^{(2)}_{01}  a^{(2)}_{10} + a^{(1)}_{10}  a^{(2)}_{02}  a^{(2)}_{10} + 
			2  (a^{(1)}_{01})^2  (a^{(2)}_{10})^2 -  a^{(1)}_{01}  a^{(2)}_{01}  (a^{(2)}_{10})^2 + 2  a^{(2)}_{02}  (a^{(2)}_{10})^2) x_1^2 x_2^3  )^T\\
			&=(x_1( (a^{(1)}_{01})^2 a^{(1)}_{10}  a^{(2)}_{10} +  a^{(1)}_{01} a^{(1)}_{10}  a^{(2)}_{01}  a^{(2)}_{10} - 
			2  (a^{(1)}_{01})^2  (a^{(2)}_{10})^2)x_1^2x_2^2, \\
			& \phantom{ -} x_2(- a^{(1)}_{01} a^{(1)}_{10}  a^{(2)}_{01}  a^{(2)}_{10} + a^{(1)}_{10}  a^{(2)}_{02}  a^{(2)}_{10} + 
			2  (a^{(1)}_{01})^2  (a^{(2)}_{10})^2 -  a^{(1)}_{01}  a^{(2)}_{01}  (a^{(2)}_{10})^2 + 2  a^{(2)}_{02}  (a^{(2)}_{10})^2)x_1^2x_2^2)^T.
		\end{align*}

		We see from the example that the expression for the normal form is "dominated" by 
		parameters in the sense that there are 12 monomials   involved parameters, whereas 
		for the phase variables we have just four  monomials $ x_1^2 x_2$, $x_1 x_2^2$,  $x_1^3 x_2^2$,   $x_1^2 x_2^3$.

		Consider polynomial equations   \eqref{ndim}
		written in the form 
		\be \label{sys_k}
		\dot x_k=  \lambda_k x_k+ x_k \sum_{{\bi} \in \Omega_k}   a^{(k)}_{\bi} x^{\bi}  \qquad (k=1,\dots, n),
		\ee
		where $\Omega_k$,  $k=1,\dots,n$, is a fixed ordered set of multi-indices ($n$-tuples) $\bi=(i_1,\dots,i_n)$, 
		whose $k$-th entry is from $\N_{-1}=\{-1\} \cup \N_0$ and all other entries are from $\N_0$. 
		The symbols  $ a_{\bi}^{(k)}$ are parameters of the family \eqref{sys_k}. Let $\Omega=\Omega_1\cup \Omega_2 \cup \dots \cup \Omega_n $.
		If $ |\Omega_k  |  $ is the number of elements of $\Omega_k$, then $\ell = |\Omega_1  |+\dots +  |\Omega_n  |$ 
		is the number of parameters   $a^{(k)}_{\bi}$ in \eqref{sys_k}.
		We will write \eqref{sys_k} in a short form as the equation 
		\be 
		\label{sys}\dot x=  A x + F(a,x),\ee
		where 
		$$
		F(a,x)=( x_1 \sum_{{\bi} \in \Omega_1} a^{(1)}_{\bi} x^{\bi}, \dots,  x_n \sum_{{\bi} \in \Omega_n} a^{(n)}_{\bi} x^{\bi}   )^T
		$$
		and $A\in \C^{n,n}$ is a diagonal matrix.

		Let 
		$\tilde L$  be the $ \ell \times n$ matrix
		which rows are all   $n$-tuples $ \bi$  from $\Omega_1,\Omega_2,\ldots,\Omega_n$ successively,   
		(that is, the rows of $\tilde L$ are the 
		subscripts $\bi$  
		of the parameters $ a^{(k)}_{\bi}$ of  \eqref{sys_k} written as row vectors).
		For  $\nu =(\nu_1,\dots, \nu_\ell) \in \N^\ell_0$ we denote by $L$  the  map
		\be \label{Lnu}
		L(\nu)=\nu  \tilde L.
		\ee
 
		Then $L(\nu)$ is the row  vector $L(\nu)=(L_1(\nu),\dots,L_n(\nu)$. Let us add that $L$ is additive map on the monoid $\N_0^\ell$, hence $L(0)=0$ and $L(j \mu)=j L(\mu)$ for any $j\in \N_0$ and  $\mu \in \N_0^\ell$.

		
		Denote by ${ a}$ the ordered (according to the order in $\Omega_k$, $k=1,2,\dots,n$) $\ell$-tuple of parameters of  equation  \eqref{sys},
		\be \label{a_vec}
		{ a}=(a^{(1)}_{\bi^{(1)}},  a^{(1)}_{\bi^{(2)}}, \dots, a^{(n)}_{\bi^{(\ell)}})
		\ee 
		and by $\C[{ a}]$ the ring of polynomials in variables $a^{(1)}_{\bi^{(1)}}, \dots,  a^{(n)}_{\bi^{(\ell)}}$
		over $\C$.
		Any monomial in parameters of system \eqref{sys} has   the form
		\be \label{nu}  a^\nu =  
		(a^{(1)}_{\bi^{(1)}})^{\nu_1}    ( a^{(1)}_{\bi^{(2)}})^{\nu_2}   \cdots (a^{(n)}_{\bi^{(\ell)}})^{\nu_\ell},
		\ee
		so, $ a^\nu$ is the abbreviation of the expression on the right hand side of \eqref{nu}. We also define the L-preimage of $\Omega$ to be
		\be \label{omega}
		\omega= L^{-1}(\Omega)
		\ee
		being non-empty since $L(e_1)=\bi^{(1)}\in\Omega$.
		
		In the case of system \eqref{ex}  we have $\Omega_1=((1,0,),(0,1),(-1,3))$, $\Omega_2=((0,2),(1,0,),(0,1))$ and $\ell =6$. Then
		$$
		\tilde L=\left(
		\begin{matrix}
			1 & 0 & -1 & 0 & 1 & 0 \\
			0 & 1 & 3 &2 & 0 & 1
		\end{matrix} 
		\right)^T
		$$
		and $
		a=(a^{(1)}_{10},   a^{(1)}_{01},   a^{(1)}_{-13},    a^{(2)}_{02},     a^{(2)}_{10},    a^{(2)}_{01})  $. For any $(\nu_1,\nu_2,\nu_3,\nu_4,\nu_5,\nu_6)\in \N_0^6$,
		$$ 
		a^\nu=
		(a^{(1)}_{10})^{\nu_1}  (a^{(1)}_{01})^{\nu_2}  (a^{(1)}_{-13})^{\nu_3}   (a^{(2)}_{02})^{\nu_4}    (a^{(2)}_{10})^{\nu_5}   (a^{(2)}_{01})^{\nu_6}   
		$$
		so, 
		$$L(\nu)= (\nu_1- \nu_3+ \nu_5, \nu_2+3 \nu_3+ 2\nu_4+ \nu_6).$$
		We see that the range of $L$, in this example, is equal to $\mathrm{im}L=\Z\times \N_0$. 
		
		In the general case,  $\mathrm{im}L$ is an additive submonoid of $\Z^n$. Moreover, by our assumptions in each column of $\tilde{L}$ at most one $-1$ can appear, and $|L(\mu)|\ge 1$ for every $\mu\in  \N_{0}^n$. Therefore, $L(\mu)=0\in \N_{-1}^n$ implies that $\mu=0$.

		\begin{defin}\label{mpoly}
			For $ m  \in  \mathrm{im}L\subseteq \Z^n  $, a  (Laurent)  polynomial $p(a)$,
			$p = \sum_{\nu\in \Supp(p)}p^{(\nu)}a^\nu$, is an $ m$-\emph{polynomial} 
			if for every  $\nu \in \Supp(p) \subset \N^\ell_{0} $, 
			$L(\nu) =  m$. The zero polynomial is said to be an $ m$-polynomial for every  $ m  \in \Z^n $.
		\end{defin}
		\textit{Remark.} Observe that by our definition, in  \eqref{ex}  every parameter   $a^{(j)}_{rs}$ 
		(that is, $a^{(1)}_{10},   a^{(1)}_{01},...$ )
		is an $(r,s)$-polynomial.
		Moreover,  the monomials  $a^{(j)}_{rs}$ have the following property: if  $a^{(j)}_{rs}$ is the 
		$q$-th element of \eqref{a_vec},  
		then $a^{(j)}_{rs}=a^{e_q}$,  $L(e_q)=(r,s)$ and  $a^{(j)}_{rs}x_1^rx_2^s=a^{e_q}x^{L(e_q)}$ ($e_q$ is the $q$-th unit basis vector of $\mathbb{F}^\ell$).    A very important observation is that all monomials $g_3$ and $g_5$ of \eqref{nfg35}, after $x_1$ and $x_2$ being factored out,  appear in the form  $a^\nu x^{L(\nu)}$.
		It is one of the cornerstones of our investigation.

		As it has just been used, $e_q \in \N_0^\ell$  will be considered as a row vector with 1 in the $q$-th position and zeroes elsewhere, and a column vector (being defined before) if being the standard basis vector in $\C^n$. 
		
		For a given $m\in\Z^n$ 
		let $ R_{ m}$  be the subset of $\C[a]$ consisting of all $ m $-polynomials. Denote by   $R$  the direct sum of   $ R_{ m}$,
		$$
		R=\oplus_{ m\in \Z^n }R_{ m}. 
		$$
		Since 
		$$
		R_{ m_1}R_{ m_2}\subseteq R_{ m_1+ m_2},
		$$
		$R$ is a graded ring.  Clearly, by the remark preceding Definiton \ref{mpoly} we have   $  R_{ 0}=\C$. Moreover, $R_{ m}$ as well as $R$ are vector spaces over $\C$ for the usual addition and multiplication by numbers from $\C$. 
		

		Let 
		$$\mathcal{M}=\{(m_1,\dots, m_n)\in \N_{-1}^n: |m|\ge 0, m\in \mathrm{im} L, m_j=-1 \text{ for at most one } j\}.$$

		Let  $M_j$ be  the space of formal vector fields of the form 
		$$
		\dot x_j=   x_j \sum_{\substack{ m\in \mathcal{M}, \\ m_i\ge 0 {\ \rm if\ } i\ne j  }} p_j^{(m)}(a) x^m   \qquad (j=1,\dots, n)
		$$
		where $ p_j^{(m)}(a)\in R_m$ for    all $j=1,\dots, n$.

		Further, let  
		$$M=M_1\times \dots \times M_n.$$
		
		Denoting $p^{(m)}(a)=\sum_{j=1}^n e_j p_j^{(m)}(a)$,  for every member of $M$ there exists a finite or infinite set $\mc L \subset \mc M$ such that it can be written as 
		\be \label{eq:M}
		\sum_{m\in {\mc L}} x \odot  p^{(m)}(a)x^m 
		\ee 
		with $\odot$ denoting the  Hadamard multiplication. It can be easily seen that $M$ is an additive group and moreover, $M$ is a module over the ring  $(\C^n,+,\odot)$ with the multiplicative unit $\bar{1}=\sum_{i=1}^n e_i \in \C^n$. Obviously, $M$ is also a vector space over $\C$.

		We note that the grading introduced in this section was implicitly used in \cite{Rom93} in order to derive an efficient 
		algorithm for computing focus (Lyapunov) quantities of planar polynomial vector fields. 
		
		\subsection{Lie brackets in $M$}
		
		We  consider elements of $M$ as (formal) vector fields, so we can compute  Lie brackets of elements of $M$.
		
		\begin{lem}
			Any element ${\bf \Theta}= (\Theta^{(1)},\dots , \Theta^{(n)})\in M $ can be written in the form
			\be \label{Theta}
			{\bf \Theta} =\sum_{\mu\in \omega} (\theta_\mu\odot x)  a^\mu x^{L(\mu)},
			\ee
			where
			$  \theta_\mu   =(  \theta_{\mu,1}  ,\dots,  \theta_{\mu,n} )$,
			$\omega$ is a finite or infinite subset of $\N^\ell_0$ such that if $\mu \in \omega$ then 
			$L(\mu)\in {\mc M}$. Additionally,  if $L_j(\mu)=-1$, then $\theta_\mu = \theta_{\mu,j} e_j$.
		\end{lem}  
		\begin{proof}
			Every  ${\bf \Theta} \in M$ can be represented in the form \eqref{eq:M}. We rewrite the polynomials $p^{(m)}(a)$, applying that each component $p^{(m)}_j(a)$ is an $m$-polynomial, as
			$$p^{(m)}(a)= \sum_{j=1}^n p_j^{(m)}(a) e_j=\sum_{j=1}^n\sum_{\mu:L(\mu)=m}\theta_{\mu,j} e_j a^\mu=\sum_{\mu:L(\mu)=m}\theta_\mu a^\mu$$
			where  	$\theta_\mu=\sum_{j=1}^n\theta_{\mu,j}e_j$. It follows that
			$${\bf \Theta}= \sum_{m\in \mathcal{L}} x\odot \sum_{\mu:L(\mu)=m}\theta_\mu a^\mu x^m= \sum_{\mu:L(\mu)\in \mathcal{L}} (  \theta_\mu\odot x)a^\mu x^{L(\mu)}$$
			where  $\omega=\{\mu: L(\mu)\in \mathcal{L}\subseteq \mc M \}$.
			
			The last assertion follows by the definition of the set $\mc M$.
		\end{proof}
		
		As every ${\bf \Theta} \in M$ is a (formal) vector field,  we can compute Lie brackets of 
		such vector fields:   
		$$[{\bf \Theta},{\bf \Phi}]:=(D{\bf \Phi}) {\bf \Theta} -  (D {\bf \Theta}) {\bf \Phi}$$
		for any ${\bf \Theta},{\bf \Phi}\in M$.\nc 
		
		\begin{lem}\label{lem_tpm}
			If ${\bf \Theta}=    (\theta_\mu \odot x)  a^\mu x^{L(\mu)}  $ and 
			$ {\bf \Phi} =  (\phi_\nu \odot  x)  a^\nu x^{L(\nu)}  $, 
			where $\mu, \nu \in \N_0^\ell$, $\theta_\mu = (  \theta_1 , \dots,  \theta_n  )^T $ 
			and   $\phi_\nu = (  \phi_1  ,\dots,  \phi_n  )^T$,   then 
			\be\label{tpm} 
			[{\bf \Theta},{\bf \Phi}]= 
			( \alpha_{\mu +\nu}\odot x)a^{\mu+\nu}  x^{L(\mu+\nu)}\in M, 
			\ee
			with 
			\be \label{alphamunu}
			\alpha_{\mu +\nu}=\la L(\nu) ,\theta_\mu \ra { \phi_\nu}    -    \la L(\mu) , {\phi_\nu}\ra  \theta_\mu.
			\ee
		\end{lem}
		\begin{proof}
			We recall that  for $f: \C^n \to \C^n$ the derivative of the Hadamard product reads
			\be \label{DHad}
			D (x \odot f(x))= {\rm diag}(f)  +{\rm diag}(x) Df(x). 
			\ee
			Let $x^{-\bar 1}=  (x_1^{-1},\dots, x_n^{-1})^T$. Using \eqref{DHad} we get
			\begin{align*}
				D(x \odot \phi_\nu\, x^{L(\nu)})&= \mathrm{diag}(\phi_\nu)x^{L(\nu)}+\mathrm{diag}(x)\phi_\nu (L(\nu)^T \odot x^{-\bar 1})^T  x^{L(\nu)}\\
				&= \mathrm{diag}(\phi_\nu)x^{L(\nu)}+ (\phi_\nu \odot x)(L(\nu)^T \odot x^{-\bar 1})^T  x^{L(\nu)}
			\end{align*}
			Taking into account \eqref{DHad}, we have 
			\begin{multline*}
				(D {\bf \Phi}){\bf \Theta} = ( D (   (\phi_\nu \odot x)  a^\nu x^{L(\nu)}  )){\bf \Theta} = ( D ( x \odot \phi_\nu x^{L(\nu)}  )  a^\nu   )) {\bf \Theta}
				=\\
				(a^\nu ({\rm diag}(\phi_\nu )  x^{L(\nu)} +
				(\phi_\nu \odot x) (L(\nu)^T \odot x^{-\bar 1})^T x^{L(\nu)})){\bf \Theta} =\\
				a^{\nu+\mu} (\phi_\nu \odot \theta_\mu \odot x)  x^{L(\nu+\mu)}  +
				a^{\nu+\mu}(\phi_\nu \odot x) (L(\nu)^T \odot x^{-\bar 1})^T  (x \odot \theta_\mu)  x^{L(\nu+\mu)} =\\
				a^{\nu+\mu}  x^{L(\nu+\mu)} (\phi_\nu \odot \theta_\mu \odot x +   \la L(\nu) , \theta_\mu\ra\phi_\nu \odot x ) 
			\end{multline*}
			and similarly,
			\be \label{e:D}
			(D {\bf \Theta}){\bf \Phi} = a^{\nu+\mu}  x^{L(\nu+\mu)} (\phi_\nu \odot \theta_\mu \odot x +   \la L(\mu) , \phi_\nu\ra\theta_\mu \odot x ).
			\ee 
			Now,  \eqref{tpm} takes place.
			
			To verify the last assertion that $[{\bf \Theta},{\bf \Phi}]\in M$ one has to recall the definition of the set $\mc M$. The  vector field    $[{\bf \Theta},{\bf \Phi}]$ would have not been in $M$ when $L(\mu +\nu)$ was not in $\mc M$. We have to consider two possibilities. Either $L_i(\mu)=L_i(\nu)=-1$ for some $i$ or, for some $i\neq j$ we have $L_i(\mu)=-1$, $L_i(\nu)=0$,  $L_j(\mu)=0$ and  $L_j(\nu)=-1$. 
			
			In the first case, $\theta_\mu =\theta_i e_i$, $\phi_\nu=\phi_i e_i$ and so, 
			$$\alpha_{\mu+\nu}=L_i(\nu) \theta_i\phi_i e_i    -     L_i(\mu) \phi_i\theta_i e_i =0.$$
			In the second case, we have $\theta_\mu =\theta_i e_i$, $\phi_\nu=\phi_j e_j$ and again, $\alpha_{\mu+\nu} =0$. 
		\end{proof}
		
		The following statement  is a direct corollary of Lemma \ref{lem_tpm}. 
		\begin{cor}
			If ${\bf \Theta}, {\bf \Phi} \in M$, then $[{\bf \Theta}, {\bf  \Phi}]\in M$.  Moreover, if 
			$${\bf \Phi}=\sum_{\nu\in \omega_1} (\phi_\nu\odot x)  a^\nu x^{L(\nu)},\ \ 
			{\bf \Theta}=\sum_{\mu\in \omega_2} (\theta_\mu\odot x)  a^\mu x^{L(\mu)},$$
			then 
			\be \label{tp}
			[{\bf \Theta}, {\bf \Phi}]=\sum_{\mu\in \omega} \sum_{\nu\in \omega_1}
			\alpha_{\mu+\nu}\  a^{\mu+\nu}  x^{L(\mu+\nu)}  .
			\ee
			with $\alpha_{\mu+\nu}$ being defined by \eqref{alphamunu}.
		\end{cor}
		Let $\mathcal{U}_s^\ell$, $s\geq 0$, be the space of polynomial vector fields of the form
		\be \label{Theta_s}
		{\bf \Theta}_s= \sum_{|\mu|=s} (\theta(\mu) \odot x ) a^\mu x^{L(\mu)}.
		\ee
		Then, obviously, $M=\oplus_{s=0}^\infty\mathcal{U}_s^\ell$.

		For any $s\in \N_0$, let
		\be \label{Us_bas}
		U_s=\{ ( \bar{1} \odot x)   a^{\mu} x^{L(\mu)} : |\mu|=s \}.
		\ee
		Since any vector field  ${\bf \Theta}_s$ of the form  \eqref{Theta_s} can be written
		as 
		$$
		{\bf \Theta}_s= \sum_{|\mu|=s} \theta(\mu) \odot (( \bar{1} \odot x)  a^\mu x^{L(\mu)}),
		$$ 
		we can treat  $\mathcal{U}_s^\ell $ as a free module generated by 
		$U_s$ over the ring $(\C^n,+,\odot)$ with the operation of multiplication by the  elements 
		of the ring being  the Hadamard product.
		The dimension of  $\mathcal{U}_s^\ell $ as the module  is the same as the dimension 
		of the space of homogeneous polynomials of degree $s$ in $\ell $ variables,    
		that is, it is 
		${\ell +s-1 \choose s}$.  
		
		\begin{defin}
			We say that ${\bf \Theta}\in M$ is \textit{of level} $s$ if ${\bf \Theta}\in  \mathcal{U}_s^\ell $
			and ${\bf \Theta}\in M$ is of \textit{the level at least}  $s$ if each term  
			$$\Theta=    (\theta \odot x)  a^\mu x^{L(\mu)}  $$
			of   $ {\bf  \Theta} $
			is in some of  $\mathcal{U}_{s+j}^\ell $, where $j\in \N_0$.
		\end{defin}

		By Lemma \ref{lem_tpm} if ${\bf \Theta} \in  \mathcal{U}_s^\ell $,
		${\bf \Phi} \in  \mathcal{U}_t^\ell $ 
		and  $ \text{ad}\  {\bf \Theta}$  is the   adjoint operator acting on $\bf \Phi$
		by
		$$
		(\text{ad\, } {\bf \Theta}) \, {\bf \Phi} =[{\bf \Theta},{\bf \Phi}],
		$$
		then  $ (\text{ad\, } {\bf \Theta})^i {\bf \Phi}$ is an element of 
		$\mathcal{U}_{is+t}^\ell $, that is,
		\be\label{ad_i}
		(\text{ad\, } {\bf \Theta})^i :  \mathcal{U}_{t}^\ell \nc  \to  \mathcal{U}_{is+t}^\ell. 
		\ee

		{\it Remark.}  The map  $(\text{ad\, } {\bf \Theta})^i $ "lifts" the space  $\mathcal{U}_{t}^\ell $ to the space 
		$  \mathcal{U}_{is+t}^\ell.$ For this reason we speak about levels in the normalization process. 
		\nc 
		
		We denote  the set of $\mu$'s appearing in \eqref{Theta_s} by $\sigma(s)$,
		$$
		\sigma(s):=\{ \mu\in \N^\ell_0 \ : |\mu|=s \}. 
		$$
		
		\subsection{The normal form algorithm}
		
		Any equation  of  the form \eqref{sys} can be written as  a   differential equation of the form 
		\be\label{Aeq}
		\dot x=a_0(x)+a_1(x)+a_2(x)+\dots = a(x), 
		\ee
		where  
		$
		a_0(x)=\lambda \odot x,
		$
		$a_s(x) \in  \mathcal{U}_s^\ell $ for $s\ge 1$, 
		$$
		a_s(x)=\sum_{\mu \in \sigma(s)}  (\alpha_\mu \odot x)  a^\mu x^{L(\mu)},
		$$
		where $\alpha_\mu\in \C^n.$
		Since  $a_0(x)  \in   \mathcal{U}_0^\ell $  and the nonlinear part  $a_1(x)$ of \eqref{sys} is from  $\mathcal{U}_1^\ell $,
		equation \eqref{sys} is of the form \eqref{Aeq}.
		
		For instance, system \eqref{ex} is written as equation \eqref{Aeq}
		with 
		$$
		a_0(x)={1 \ch -1} \odot {x_1 \ch x_2}, 
		$$
		$$
		\begin{aligned}
			a_1(x)= & {1 \ch 0 } \odot {x_1 \ch x_2} a^{(1)}_{10} x_1 + {1 \ch 0 } \odot {x_1 \ch x_2}  a^{(1)}_{01} x_2+  {1 \ch 0 } \odot {x_1 \ch x_2}  a^{(1)}_{-13} x_1^{-1} x_2^3+\\
			&  {0 \ch 1 } \odot {x_1 \ch x_2}  a^{(2)}_{10} x_1 + {0 \ch 1 } \odot {x_1 \ch x_2}  a^{(2)}_{01} x_2+  {0 \ch 1 } \odot {x_1 \ch x_2}  a^{(2)}_{02}  x_2^2.
		\end{aligned}
		$$
	    Note that \eqref{sys_k} is just a special case of \eqref{Aeq}.
				
		Below we will present an algorithm which normalizes equation \eqref{Aeq} up to a certain  level.  

		\begin{lem}\label{lem2}
			If  
			\be \label{re_t}
			\la L(\mu), \lambda \ra =0, 
			\ee
			then all entries of the vector   field 
			\be \label{tmu}
			(\alpha_\mu \odot x)  a^\mu x^{L(\mu)}
			\ee
			are resonant terms.  
		\end{lem}
		\begin{proof}
			The $i$-th entry of \eqref{tmu} is
			$$  \alpha^{(\mu)}_i x_i  x^{L(\mu)} a^\mu =  \alpha^{(\mu)}_i   x_1^{L_1(\mu)}\cdots  x_{i-1}^{L_{i-1}(\mu)} 
			x_{i}^{L_{i}(\mu)+1}  x_{i+1}^{L_{i+1}(\mu)} \cdots     x_{n}^{L_{n}(\mu)} a^\mu.     $$ 
			Since $$ \la {L_1(\mu)}, \dots,    {L_{i-1}(\mu)},  
			{L_{i}(\mu)+1},  {L_{i+1}(\mu)}, \dots ,{L_{n}(\mu)}, \lambda\ra -\lambda_i  =
			\la L(\mu), \lambda\ra,$$
			the statement of the lemma takes place.  
		\end{proof}
		
		Lemma \ref{lem2} justifies the following definition.
		\begin{defin}
			It is said  that a term of the form \eqref{tmu}
			of the right-hand side of \eqref{Aeq} is 
			\emph{resonant} if  \eqref{re_t} holds. 
		\end{defin}
		
		Under known approaches   a normalization is  performed up to terms of some degree of $x$.
		However, it is possible to perform the normalization up to some degree of the polynomials in  
		the parameters of the system. 
		
		\begin{defin}
			We say that equation \eqref{Aeq}  is\textit{ in the normal form up to level} $s$ if 
			all non-resonant terms  in $a_1(x),\dots, a_{s}(x)$ are equal to zero.  
		\end{defin}
		
		Clearly, if equation \eqref{Aeq} is in the normal form for all levels $s\in \N$, 
		then it is in the normal form in the sense of Definition \ref{def_nf}.

		By definition we say that the operator   $\pounds^a : M \to M$  which acts on elements of the form  \eqref{Theta} by 
		\be \label{La}
		\pounds^a ( \sum_{\mu \in \omega}
		(\theta_\mu \odot x)  a^\mu x^{L(\mu)})=  \sum_{\mu\in \omega} \la L(\mu), \lambda \ra  (\theta_\mu \odot x)  a^\mu x^{L(\mu)}) 
		\ee
		is the {\it homological operator } of equation \eqref{Aeq} (the superscript $a$ indicates that the operator 
		is related to the equation \eqref{Aeq} with the right-hand sides being an element of the module $M$, 
		so it is different from the homological operator  of  \eqref{ndim}  defined by  \eqref{ho}.
		It is not difficult to check that $\pounds^a$ is a homomorphism of the module $M$ over $\C^n$. 
		Clearly being additive, we also have that for any $\phi\in\C^n$
		\begin{align*} 
			\pounds^a (\phi \odot  \sum_{\mu \in \omega}
			( \theta{(\mu)} \odot x)  a^\mu x^{L(\mu)})&=  \sum_{\mu\in \omega} \la L(\mu), \lambda \ra  (\phi \odot\theta{(\mu)} \odot x)  a^\mu x^{L(\mu)}\\
			&=\phi \odot \pounds^a ( \sum_{\mu \in \omega}
			( \theta{(\mu)} \odot x)  a^\mu x^{L(\mu)}) ) .
		\end{align*}
		
		The restriction of 
		$
		\pounds^a$ on   $\mathcal{U}_s^\ell $ is denoted by $\pounds^a_s$. 
		Obviously,  $
		\pounds^a_s: \mathcal{U}_s^\ell  \to  \mathcal{U}_s^\ell $. 
		From \eqref{La} it is clear that the set $U_s$ defined 
		by \eqref{Us_bas} is the set of basis eigenvectors of $\pounds^a_s$ 
		and 
		we can split $  \mathcal{U}_s^\ell $   as
		$$
		\mathcal{U}_s^\ell   ={\rm im} \pounds_s^a \oplus \ker \pounds_s^a.
		$$
		
		Assume  that equation \eqref{Aeq} is in the normal form up to level $s-1$, {$s\ge 1$},
		that is,  for terms of the  form \eqref{tmu} appearing in \eqref{Aeq} if 
		$|\mu|\leq s-1$ and $\la L(\mu), \lambda\ra \ne 0 $,  then $  \alpha{(\mu)}=0$.
		Then the homological equation 
		$$
		\pounds^a_s(h_s)=a_s-g_s 
		$$
		can be solved  for $h_s$ and $g_s$  as follows:
		\be \label{hsd}
		h_s(x)= \sum_{\substack{\mu:\mu \in \sigma(s),\\ \la L(\mu), \lambda \ra \ne 0}} \frac 1{  \la L(\mu), \lambda \ra  }  (\alpha_\mu \odot x)  a^\mu x^{L(\mu)}, 
		\ee
		\be \label{gsd}
		g_s(x)= \sum_{\substack{\mu:\mu \in \sigma(s),\\ \la L(\mu), \lambda \ra= 0}}   (\alpha_\mu \odot x)  a^\mu x^{L(\mu)}.
		\ee

		\begin{teo}\label{th_nf} Assume that equation  \eqref{Aeq} is in the normal form up to level $s-1$, $s\ge 1$, and 
			let 
			\be \label{Hs}
			H_s(x)=\exp (h_s(x)),
			\ee 
			where $h_s$ is defined by \eqref{hsd}. 
			Then performing the   substitution 
			$y=H_s(x)$ and changing $y$ to $x$ we obtain from \eqref{Aeq} an equation,    which  the right hand side is from $M$ and is in the 
			normal form up to level $s$.
		\end{teo}
		\begin{proof}
			According to the theorem on Lie series for vector fields after transformation \eqref{Hs} we obtain from \eqref{sys} the vector field 
			\be \label{ahs}
			a(x) +  (\text{ad\ }h_s)  a(x)+  \sum_{i=2}^\infty \frac{1}{i!} (\text{ad\ }h_s)^i a(x).
			\ee
			By \eqref{ad_i} the last summand  is of the level at least $s+1$ 
			and for the first two, we have 
			\begin{multline*}
				a(x) +  (\text{ad\ }h_s)  a(x)= a(x)+ [h_s(x), a(x)]=\\ a_0(x)+a_1(x)+\dots + a_{s-1}(x)
				+a_s(x)+[h_s(x),a_0(x)]+\dots,  
			\end{multline*}
			where the dots stand for the terms of the level at least $s+1$. 
			
			By \eqref{tpm}, \eqref{hsd} and \eqref{gsd} $a_s(x)+ [h_s(x),a_0(x)]=g_s(x)$.
			Since by our assumption \eqref{Aeq} is in the normal form up to level $s-1$,
			the  expression \eqref{ahs} is in the normal form up to level $s$. 
		\end{proof}
		
		As a direct outcome of Theorem \ref{th_nf} we have the following statement. 
		\begin{cor}
			There are polynomial maps $H_1(x), \dots, H_s(x)$, such that equation  \eqref{sys}
			is transformed to an equation which is in the normal form up to level $s$ by the transformation 
			$y=H_s\circ \cdots \circ H_1$. 
		\end{cor}
		
		From the results presented above we have the following algorithm for the normalization of equation 
		\eqref{Aeq} up to level  $m$. 
		
		\medskip
		
		{\it Algorithm A.} 
		
		Set $a_0(x):=A x$,
		${\bf \Xi}_k(x):=0$ for $k=1,2,\dots, m.$ 
		
		For $s=1,\dots,m$ do the following:
		\begin{enumerate}	
			\item[ (i)] Define $h_s$ and $g_s$ by \eqref{hsd} and \eqref{gsd}; 
			\item[ (ii)] Compute 
			$$
			{\bf \Xi}= \sum_{k=0}^{s-1} \sum_{  i=1}^{ \lfloor \frac{m-k}{s}\rfloor }  \frac{1}{i!} (\text{ad\ }h_s(x))^i a_k(x)
			$$ 
			and write 
			${\bf \Xi}= \sum_{i=s+1}^m {\bf \Xi}_i,   $ where ${\bf \Xi}_i\in \mathcal{U}^\ell_i$;
			\item[ (iii)] Let $a_s=g_s$, $a_{i}={\bf \Xi}_i$ for $i=s+1,\dots, m$.	  
		\end{enumerate} 
		\medskip

		We see that the  algorithm is different from  Algorithm 1, in particular, the calculations are  performed in different spaces. 
		In the next section, we reformulate the algorithm in a simpler setting and give an example of 
		computing the normal form of equation \eqref{ex}.  
		
		\begin{pro} \label{pro_nf_k}
			If system   \eqref{Aeq} is in normal form up to level $s$, then it is in the normal form up to 
			order at least $s+1$.
		\end{pro}
		\begin{proof}
			For each multiindex $\bi  \in \Omega$ it holds  $|\bi|\ge 1$.  Therefore, by \eqref{Lnu},  if a term of the form $(\theta \odot x)a^\mu x^{L(\mu)}$ is of 
			level $s$ then $|L(\mu)| \ge s$. Thus, all resonant terms in the normal form  of level $s$ or higher are  of order $s+1$ or higher. 
		\end{proof}
		
		\section{Generalized  vector fields}
		
		We treat $\N_0^\ell $ as an {\it ordered set} (ordered, for instance, using  the degree lexicographic order).
		
		\begin{defin} 
			Let   $ \alpha $  be a map  defined on some subset $\omega $ of $\N_0^\ell $ 
			$$
			\alpha : \omega\subset  \N_0^\ell \to \C^n,
			$$ 
			that is,  $ \alpha$ assigns to every $\nu\in  \omega  $ an $n$-tuple  
			$$ \alpha_\nu=( \alpha_1{(\nu)},\dots,  \alpha_n{(\nu)}).$$
			We say that an $n$-tuple  
			of formal power series 
			\be \label{al_dis}
			\hat \alpha=\sum_{\nu \in \omega} \alpha_\nu a^\nu,
			\ee
			where  $\omega=\text{Supp}(\hat \alpha)$, 
			is a \emph{generalized formal vector field}.
		\end{defin}
		\nc

		Note that  expression \eqref{al_dis} looks similar to a usual  vector field depending on variables  
		$a^{(1)}_{\bi^{(1)}}),     a^{(1)}_{\bi^{(2)}},   \dots, a^{(n)}_{\bi^{(\ell_n)}}$
		defined 
		on  $\C^\ell$, 
		however,  it is not a vector  field in these variables  in the usual sense, because the usual vector field is defined by assigning  
		to a vector from  $\C^\ell$ a vector of the same dimension, but  if a  series \eqref{al_dis}
		converges it  assigns to  a point from  $\C^\ell$ a vector from  $\C^n$. 
		
		Since  the objects defined above are not vector fields in the usual sense, we 
		call them generalized vector fields.  
		\nc  
		
		We  denote the  set  of all formal vector fields defined by \eqref{al_dis} by  $\mathcal{A} $. 
		It is not difficult to check that $\mathcal{A} $ is a module over the ring $(\C^n,+,\odot)$.
		The zero vector in $\mathcal{A}$ is a series \eqref{al_dis} where $\alpha_\nu=0\in \C^n$  for all 
		$\nu\in \N_0^\ell$.   For any $k\in\N_0$, let  $\mathcal{A}_k$ be the subset of all elements of $\mathcal{A}$ of 
		the form 
		$$
		\sum_{\mu: |\mu|=k}  \alpha_\mu a^\mu.
		$$
		Then   $\mathcal{A}_k$ is a  module over the ring $(\C^n,+,\odot)$ and  $\mathcal{A}$ is a direct sum of   $\mathcal{A}_k$, $k=0,1,\dots$.

		Recall that we  consider  $M$ as the direct sum of modules $\mathcal{U}^\ell_s $ over $\C^n$, $s=0,1,2,\dots$,   and  define a module homomorphism
		$$
		\mathfrak {T} :  \mathcal{A} \to M
		$$
		\be \label{T}
		\mathfrak{T} \left( \sum_{\mu \in \omega}  \theta_\mu a^\mu \right)=\sum_{\mu\in \omega} (\theta_\mu\odot x)  a^\mu x^{L(\mu)}.
		\ee
		Clearly,  $  \mathfrak{T}$ is an  isomorphism. 
		
		We define the Lie bracket of $\hat{\theta}= \sum_{\mu \in \omega}  \theta_\mu a^\mu$ and 
		$  \hat  \phi= \sum_{\nu \in \omega_1} \phi_\nu a^\nu$  by
		$$
		[\hat \theta,\hat \phi]= \mathfrak{T}^{-1}( [\mathfrak{T}(\hat\theta) , \mathfrak{T}(\hat\phi)]). 
		$$
		Then, by  \eqref{tp} and \eqref{T}  
		\be \label{lie_di} 
		[\hat  \theta , \hat \phi]=\sum_{\mu\in \omega} \sum_{\nu\in \omega_1}   
		\left( \la L(\nu) , \theta_\mu \ra  \phi_\nu  -  \la L(\mu) , \phi_\nu\ra  \theta_\mu     
		\right) a^{\mu+\nu}.
		\ee 
		Since 
		\begin{equation*}
			\mathfrak{T}\left([\hat \psi,[\hat \theta,\hat \phi]]\right)= [ \mathfrak{T}(\hat \psi),  \mathfrak{T}([\hat \theta , \hat \phi])] = [ \mathfrak{T}(\hat \psi),  [\mathfrak{T}(\hat \theta) , \mathfrak{T}(\hat \phi)] ) ] , 
		\end{equation*}
		it is easily seen 
		that for the Lie bracket  in $\mathcal{A}$  defined by \eqref{lie_di} the
		Jacobi identity holds, so  $\mathcal{A}$, {being also a vector space over $\C$}, is a Lie algebra
		and so $\mathfrak{T}$ defines a Lie algebra isomorphism. 
		
		We could have defined the Lie bracket in $\mathcal{A}$ in a different way. The $n\times n$ matrix  $\mathfrak{D}(\theta_\mu a^\mu):=\theta_\mu L(\mu)a^\mu$, associated to the monomial $\theta_\mu a^\mu$, 
		can be extended to a $\C^n$-module homomorphism (or linear operator) from $\C^n$ to $\C^n$ inducing the Lie bracket. Indeed, the setting 
		\[  [\hat  \theta , \hat \phi] = \mathfrak{D}(\phi_\nu a^\nu)\theta_\mu a^\mu - \mathfrak{D}(\theta_\mu a^\mu)\phi_\nu a^\nu  \]
		produces the same as \eqref{lie_di} which can be observed from 
		\begin{align*} \mathfrak{D}(\phi_\nu a^\nu)\theta_\mu a^\mu - \mathfrak{D}(\theta_\mu a^\mu)\phi_\nu a^\nu
			&= \phi_\nu L(\nu)\theta_\mu a^{\mu+\nu}-\theta_\mu L(\mu)\phi_\nu a^{\mu+\nu} \\
			&= \left( \la L(\nu),\theta_\mu \ra  \phi_\nu - \la L(\mu)\phi_\nu \ra \theta_\mu \right) a^{\mu+\nu}.
		\end{align*}

		One may expect that this operator is  the derivative  $D(\hat\Theta)$ evaluated at $x=\bar{1}$ but it is not the case. However, the connection exists,  the second term of \eqref{e:D} evaluated at $x=\bar{1}$, coincides with $\mathfrak{D}(\theta_\mu a^\mu)$. So, we may think of $\mathfrak{D}$ as a derivative.
		
		We will compute normal forms in terms of such generalized  vector fields in the following way.  
		
		Let $\hat \alpha$ be the image of the right hand side of \eqref{Aeq} under the isomorphism $\mathfrak{T}^{-1}$, so
		$$
		\mathfrak{T}^{-1}(a(x))=\sum_{k=0}^\infty \hat \alpha_k,
		$$
		where 
		\be \label{als}
		\hat \alpha_0=\lambda, \quad 
		\hat \alpha_k=\sum_{\mu \in \sigma(k)} \alpha_\mu a^\mu \quad \text{for}\  k\ge 1 .
		\ee
		\begin{defin}
			It is said that the generalized  vector field  $\hat \alpha =\sum_{k=0}^\infty \hat \alpha_k  $  (where $\hat \alpha_k $ is of the form \eqref{als})  is in\textit{ the normal form up to level} $s$ if the coefficients 
			of  all non-resonant terms   in $   \hat \alpha_1, \dots,\hat  \alpha_s$ are equal to zero. The generalized  vector field $\hat \alpha _0$ will be always considered to be in the  normal form  up to level $0$.
		\end{defin}
		
		Assume that $\hat \alpha$  is in the normal form up to level $s-1$, $s\ge 1$. Let
		\be  \label{ez}
		\begin{aligned}
			\hat\eta_s&= \sum_{\substack{\mu:\mu \in \sigma(s),\\ \la L(\mu), \lambda \ra \ne 0}} \frac 1{  \la L(\mu), \lambda \ra  }  \alpha_\mu   a^\mu ,  \\
			\hat\zeta_s&= \sum_{\substack{\mu:\mu \in \sigma(s),\\ \la L(\mu), \lambda \ra= 0}}   \alpha_\mu   a^\mu.
		\end{aligned}
		\ee
		that is $\hat \eta_s=\mathfrak{T}^{-1}(h_s(x)),$ and   $\hat \zeta_s=\mathfrak{T}^{-1}(g_s(x)),$ 
		where $h_s(x)$ and $g_s(x)$ are defined by \eqref{hsd} and \eqref{gsd}, respectively. 
		Since $\mathfrak{T}$ is a Lie algebra isomorphism we can reformulate Algorithm A as follows.  
		
		\medskip

		{\it Algorithm B.}
		
		Set $\hat \alpha_0:= \lambda$,  
		$\hat \xi_k:=0$ 
		for $k=1,2,\dots, m$. 
		
		For $s=1,\dots,m$ do the following:
		\begin{enumerate}
			\item[ (i)] Define $\hat \eta_s$ and $\hat \zeta_s$ by \eqref{ez}; 
			
			\item[ (ii)] Compute 
			$$
			\hat \xi =\sum_{k=0}^{s-1} \sum_{  i=1}^{\lfloor \frac{m-k}{s}\rfloor }   \frac{1}{i!} (\text{ad\ }\hat \eta_s)^i\hat  \alpha_k  
			$$ 
			(where $\ad \hat \eta_s:=[\hat \eta_s,\cdot]$ is the adjoint operator acting on $\mathcal{A}$) 
			and represent $\hat \xi$ in the form  $\hat \xi=\sum_{i=s}^m \hat \xi_i$, where $\hat \xi_i\in \widehat{U}_i^\ell$;
			\item[ (iii)] Let $\hat \alpha_s=\hat \zeta_s$, $\hat \alpha_{s+1}=\hat \xi_{s+1}, \dots, \hat \alpha_{m}=\hat \xi_{m}$. 
		\end{enumerate} 
		
		The obtained vector field  $\hat \alpha$ is in the normal form up to level $m$.  
		
		\medskip
		
		By Proposition \ref{pro_nf_k} 
		if  system \eqref{sys}   is in normal form up to level $s$ then it is in the normal form up to 
		order at least $s+1$. Obviously, when    $\hat \alpha$ is in the normal form up to level $m$,
		a normal form of \eqref{sys} can be built up from   $\hat \alpha$ using the following procedure:
		
		\medskip
		
		Set $g_k(x)= (0,0,\dots,0)^T$  for $k=1,\dots, m$.  For $\mu \in \cup_{k=1}^m{\sigma(k)} $ do the following:\\
		if  $\alpha_\mu \neq  0, \ |L(\mu)|=k $, then $g_k(x)= g_k(x)+(\alpha_\mu \odot x)  a^\mu x^{L(\mu)}$. 
		
		\medskip
		
		Clearly, the obtained equation 
		$$
		\dot x= A x+\sum_{k=1}^m g_k(x)
		$$
		is the normalization of \eqref{sys} up to order $m$. 
		
		\medskip
		
		{\it Example.}  As an example of applying the algorithm we consider the computation of normal form of system \eqref{ex}.  For $\mu \in \N_0^\ell$ we will use the abbreviation $ [\mu]=[\mu_1,\dots, \mu_\ell] := a^\mu $.
		
		By Proposition \ref{pro_nf_k} in order to compute the normal form of \eqref{ex} up to order 5 it is 
		sufficient to compute the normal form of \eqref{ex} up to level 4. 
		
		At the level $0$ the set $\sigma(0) $ consists of only one vector, $(0,0,0,0,0,0)$ with 
		$$\hat \alpha_0=\alpha_{(  0,0,0,0,0,0)}={1 \ch -1 }[0,0,0,0,0,0].$$
		Passing to the level 1,   $\sigma(1)$ is the set of vectors   
		\be \label{eT}
		e_1^T, \dots, e_6^T,
		\ee
		which form the standard basis of $\Z^6$.  The 
		vector field $\hat \alpha_1$ is obtained by using the nonlinear terms of \eqref{ex}:
		\begin{align*}
			\hat \alpha_1&={1 \ch 0 }[e_1^T]+ {1 \ch 0 } [e_2^T]+{1 \ch 0 }[e_3^T]+ {0 \ch 1 }[e_4^T]+ {0 \ch 1 } [e_5^T]+{0 \ch 1 }[e_6^T]\\
			&={1 \ch 0 } a^{(1)}_{10} + {1 \ch 0 }  a^{(1)}_{01} +{1 \ch 0 }  a^{(1)}_{-13}+ {0 \ch 1 }  a^{(2)}_{02}+ {0 \ch 1 }  a^{(2)}_{10} +{0 \ch 1 } a^{(1)}_{01}.
		\end{align*}
		Then by (i) of Algorithm B for $s=1$ we have $\zeta_1=\bar 0$ and 
		\begin{align*}
			\hat \eta_1&=   {1 \ch 0 } [e_1^T] - {1 \ch 0 } [e_2^T]-  \frac 14 {1 \ch 0 } [e_3^T]- \frac 12 {0 \ch 1 } [e_4^T]+
			{0 \ch 1 } [e_5^T] -{0 \ch 1 }[e_6^T]
			\\ 
			&= {1 \ch 0 } a^{(1)}_{10} - {1 \ch 0 }  a^{(1)}_{01}-  \frac 14 {1 \ch 0 }  a^{(1)}_{-13}- \frac 12 {0 \ch 1 }  a^{(2)}_{02}+ {0 \ch 1 }  a^{(2)}_{10} -{0 \ch 1 } a^{(1)}_{01}.
		\end{align*}
		
		When we know the level $s$, that is the set  $\sigma(s)$, the next level, the set  $\sigma(s+1)$, is obtained
		by adding to the elements of $\sigma(s)$, one of vectors \eqref{eT}.

		According to the  Algorithm B we have to  set 
		\be\label{in_con} \hat \alpha_2=\hat \alpha_3=\hat \alpha_4=(0,0,\dots,0)^T.
		\ee
		
		Next, we compute $\hat\xi_1+\dots +\hat \xi_4$ according to (ii); that is, we compute the sum 
		\be \label{triang}
		\begin{aligned}
			& \phantom{ \hat \alpha_0+ (\ad \hat \eta_1) \hat \alpha_0  +} \frac 12  (\ad \hat \eta_1)^2 \hat \alpha_0+\frac 1{3!} 
			(\ad \hat \eta_1)^3 \hat \alpha_0+  \frac 1{4!} 
			(\ad \hat \eta_1)^4 \hat \alpha_0+ \\
			& \phantom{\hat \alpha_1+ }(\ad \hat \eta_1) \hat \alpha_1  +\frac 12  (\ad \hat \eta_1)^2 \hat \alpha_1+\frac 1{3!} 
			(\ad \hat \eta_1)^3 \hat \alpha_1+\\
			&
			\hat \alpha_2+  (\ad \hat \eta_1) \hat \alpha_2  +\frac 12  (\ad \hat \eta_1)^2 \hat \alpha_2 +\\
			&  \hat \alpha_3 + (\ad \hat \eta_1) \hat \alpha_3  +\\
			&  \hat \alpha_4.\\
		\end{aligned}
		\ee
		Then, for level 1, we have   
		\be \label{lev1} 
		(\ad \hat \eta_1) \hat \alpha_0 +\hat \alpha_1=\bar 0.
		\ee   		
  One can observe that the terms on the diagonals of \eqref{triang} are  at  the same levels, 
that is, they are  from $\mathcal{U}^6_2,  \mathcal{U}^6_3 $ 	and $	\mathcal{U}^6_4$, respectively. 
		Thus, for level 2, we obtain
		\be \label{lev2}
		\hat \alpha_2+ (\ad \hat \eta_1) \hat \alpha_1 +\frac 12  (\ad \hat \eta_1)^2 \hat \alpha_0.
		\ee 
		By \eqref{in_con} $  \hat \alpha_2= 0$ and for the second  term of the sum given above
		$$
		\begin{aligned}
			(\ad \hat \eta_1) \hat \alpha_1 =& {-2 \ch 0} [1,1,0,0,0,0]+
			{-5/2\choose 0} [1,0,1,0,0,0]+{3/4 \ch 0} [0,1,1,0,0,0]+\\
			& {1/2\ch 0}[0,1,0,1,0,0]
			+
			{2\ch -2}[0,1,0,0,1,0]+ {-3/4 \ch 0} [ 0,0,1,1,0,0]+\\
			&{15/4\ch -5/4}[0,0,1,0,1,0]
			+
			{-9/4\ch 0}[0,0,1,0,0,1]+
			{0\ch 3}[0,0,0,1,1,0]
			+\\
			&
			{0\ch -1/2}[0,0,0,1,0,1]
			+
			{0\ch 2}[0,0,0,0,1,1].
		\end{aligned} 
		$$
		From \eqref{lev1}  we observe that  $(\ad \hat \eta_1)\,  \hat \alpha_0=-\hat \alpha_1$ and using this in \eqref{lev2} gives
		
		\begin{align*}
			\frac{1}{2}  (\text{ad\ } \eta_1)^2 \alpha_0+   (\text{ad\ }\eta_1) \alpha_1&= 
			{-1 \ch 0 }[1,1,0,0,0,0]
			+
			{-5/4\ch 0 }[1,0,1,0,0,0] \\
			&+
			{3/8\ch 0 }[0,1,1,0,0,0]
			+
			{1/4 \ch 0 }[0,1,0,1,0,0]\\
			&+
			{1\ch  -1 }[0,1,0,0,1,0]
			+
			{-3/8 \ch 0 }[0,0,1,1,0,0]\\
			&+
			{15/8 \ch -5/8 }[0,0,1,0,1,0]
			+
			{-9/8 \ch 0 }[0,0,1,0,0,1]\\
			&+
			{0 \ch 3/2 }[0,0,0,1,1,0]
			+
			{0\ch  -1/4 }[0,0,0,1,0,1]\\
			&+
			{0 \ch 1}[0,0,0,0,1,1].
		\end{align*}
		Continuing  the computations following Algorithm B  we obtain 
		\begin{align*}
			\hat \alpha_1&=0;  \\
			\hat\alpha_2&=  \left(- a^{(1)}_{01} a^{(1)}_{10}  +  a^{(1)}_{01}  a^{(2)}_{10} , - a^{(1)}_{01}  a^{(2)}_{10}  + 
			a^{(2)}_{01}  a^{(2)}_{10}\right)^T; \\
			\hat\alpha_3&= \left(0, a^{(1)}_{10}  a^{(2)}_{02}  a^{(2)}_{10} + 2  a^{(2)}_{02}  (a^{(2)}_{10})^2\right)^T; \\
			\hat\alpha_4&= \left( (a^{(1)}_{01})^2 a^{(1)}_{10}  a^{(2)}_{10} +  a^{(1)}_{01} a^{(1)}_{10}  a^{(2)}_{01}  a^{(2)}_{10} - 2  (a^{(1)}_{01})^2  (a^{(2)}_{10})^2\right., \\
			& \left.- a^{(1)}_{01} a^{(1)}_{10}  a^{(2)}_{01}  a^{(2)}_{10} + 2  (a^{(1)}_{01})^2  (a^{(2)}_{10})^2 -   a^{(2)}_{01}  (a^{(2)}_{10})^2 \right)^T.
		\end{align*}
		so, the normal form up to level 4 is
		\be \label{nf4}
		\hat \alpha= \hat \alpha_0+ \hat \alpha_1+\hat \alpha_2+\hat \alpha_3+\hat \alpha_4.
		\ee

{\it Remark.} From Algorithm A it is obvious that in the case of equation \eqref{sys}, where $F(a,x)\in \vnj$ 
the normal form produced by Algorithm A is the same as the usual   Poincar\'e-Dulac   normal form
(produced e.g. algorithms of \cite{Mur,San,Wal}).  	Normal form \eqref{nf4} agrees with \eqref{nfg35}.
		However, computing up to  the seventh order we see that Algorithm B and the algorithms of \cite{Mur,San,Wal}
		produce normal forms  which differ in terms of order 7.

		A  great advantage of Algorithm B is that it allows parallel computations of terms of  normal forms, that is, 
		it is possible to compute one term of the normal form at once based on the following statement.
		
		\begin{pro} \label{Pr2}
			If for some  $\kappa\in \N_0^\ell$ we have  $\la L(\kappa), \lambda\ra=0$, then the coefficient  $\alpha_\kappa$ 
			in the normal  form  $\hat  \alpha$ is computed using only the  terms   $\alpha_\mu a^\mu$  such that 
			$a^\mu$ divides $a^\kappa$, i.e. $\mu_j \le \kappa_j$ for all $j=1,2,\dots, \ell$.  
		\end{pro}
		\begin{proof}
			The correctness of the statements follows from \eqref{lie_di} and Algorithm B.
		\end{proof}
		
		Thus, in order to compute a coefficient $\alpha_\kappa$ of the normal form, we 
		first look for the set of $\mu$'s involved in the computation of  $\alpha_\kappa$ by means of Algorithm B. 
		Denote this set $\omega_\kappa$. The set   $\omega_\kappa$ can be found using, for instance,
		the following procedure:
		
		Let $|\kappa|=s$. Set $p=1$, $\tau_\kappa(s)=\{ \kappa \}.$  
		
		While $p<s$ do \\ \indent \indent
		set  $\tau_\kappa(s-p)  =\emptyset$; \\
		\indent \indent
		for $ \mu=(\mu_1, \dots, \mu_\ell  ):$ \\ 
		\indent \indent\indent for $i=1, \dots, \ell $:
		if $\mu_i-i\ge 0$ then  $\tau_\kappa(s-p)= \tau_\kappa(k)\cup \{\mu \}$;\\
		\indent \indent set $p=p+1$. 
		\medskip
		
		The output of the procedure is the sequence of the sets $\tau_\kappa(i)$, $i=1, \dots, s-1$, 
		where $\tau_\kappa(i)$ is a subset of elements of level $i$.  Then
		\begin{equation*}
			\tau_\kappa=\cup_{i=1}^{s-1} \tau_\kappa(i)
		\end{equation*} 
		is the   subset of $\N_0^\ell$ needed in the computation of $\alpha_\kappa$
		and, in order to compute the  $\alpha_\kappa$, one just uses Algorithm B, where the Lie brackets are computed 
		with $\omega_1$ and $\omega_2$ 
		in \eqref{lie_di} being subsets of $\tau_\kappa$.

		Thus, based on Proposition \ref{Pr2} and the procedure described above, one can easily parallelize normal form computations.

		\section*{Acknowledgments}
		The first author is supported by the Slovenian Research Agency (core research program P1-0288) and the second author is supported by the Slovenian Research Agency (core research program P1-0306).
		
		We acknowledge COST (European Cooperation in Science and Technology)
		actions CA15140 (Improving Applicability of Nature-Inspired Optimisation
		by Joining Theory and Practice (ImAppNIO)) and IC1406 (High-Performance Modelling and Simulation for Big Data Applications (cHiPSet))
		and Tomas Bata University in Zl\'in for accessing Wolfram Mathematica.

	\end{document}